\documentclass[12pt,a4paper]{amsart}
\usepackage{amsmath}
\usepackage{amsthm}
\usepackage{amsfonts}
\usepackage{amssymb}
\usepackage{dsfont}
\usepackage{xr}

\newlength{\abstand}
\catcode`\ä=\active
\def ä{\"a}
\catcode`\ö=\active
\def ö{\"o}
\catcode`\ü=\active
\def ü{\"u}
\catcode`\Ä=\active
\def Ä{\"A}
\catcode`\Ö=\active
\def Ö{\"O}
\catcode`\Ü=\active
\def Ü{\"U}
\catcode`\ß=\active
\def ß{\ss{}}
\catcode`\é=\active
\def é{\'{e}}
\catcode`\É=\active
\def É{\'{E}}

\DeclareFontFamily{OT1}{pzc}{}
\DeclareFontShape{OT1}{pzc}{m}{it}{<-> s * [1.150] pzcmi7t}{}
\DeclareMathAlphabet{\mathcal}{OT1}{pzc}{m}{it}

\def\A{{\mathds A}}
\def\F{{\mathds F}}

\def\N{{\mathds N}}
\def\Nn{\N_0}
\def\Np{\N_+}

\def\Nns{\str{\Nn}}
\def\Nps{\str{\Np}}

\def\cA{\mathcal A}
\def\cB{\mathcal B}
\def\cC{\mathcal C}
\def\cD{\mathcal D}
\def\cE{\mathcal E}

\def\cM{\mathcal M}

\def\cU{\mathcal U}

\def\res{\mathrm{res}}

\def\Sets{\mathcal{Sets}}
\def\SSets{\mbox{$\Delta$-$\Sets$}}
\def\SSetsb{\SSets^+}
\def\sSSets{\str{\SSets}}
\def\sSSetsb{\str{\SSetsb}}

\def\SCH{\mathcal{Sch}}

\def\COH{\mathcal{Coh}}

\def\CAT{\mathcal{Cat}}
\def\MCAT{\mathcal{ModCat}}
\def\sMCAT{\str{\MCAT}}
\def\ECAT{\mathcal{ExCat}}
\def\sECAT{\str{\ECAT}}
\def\Sing{\mathcal{Sing}}
\def\Top{\mathcal{Top}}

\newcommand\str[1]{{\mbox{}^*#1}}
\newcommand\MorC[1]{\mathrm{Mor}_{#1}}
\newcommand\Mor[3]{\MorC{#1}(#2,#3)}
\newcommand\id[1]{\mathds{1}_{#1}}

\newcommand\Schfp[1]{\SCH^{\mathrm{fp}}_{#1}}

\newcommand\op[1]{{#1}^{\mathrm{op}}}

\newcommand\Coh[1]{\COH_{#1}}
\newcommand\sCoh[1]{\str{\COH_{#1}}}

\newcommand\s[1]{N\,#1}
\newcommand\sa[1]{\tilde{N}\,#1}

\newcommand\K[2]{\mathrm{K}_{#1}(#2)}

\newcommand\sK[2]{\str{\mathrm{K}}_{#1}(#2)}

\newcommand\Cat[1]{\mbox{${#1}$-$\mathcal{Cat}$}}
\newcommand\UCat[2]{\Cat{#2}^{#1}}
\newcommand\sCat[1]{\mbox{${#1}$-$^*\mathcal{Cat}$}}
\newcommand\Ho[1]{\mathrm{Ho(#1)}}
\newcommand\sHo[1]{\str{\Ho{#1}}}
\newcommand\Qu[1]{\mathrm{Q{#1}}}
\newcommand\sQu[1]{\str{\Qu{#1}}}
\newcommand\Ne[1]{\mathrm{N{#1}}}
\newcommand\sNe[1]{\str{\Ne{#1}}}
\newcommand\USSets[1]{\SSets^{#1}}
\newcommand\Ex[1]{\mathrm{Ex}^\infty{#1}}
\newcommand\ex[1]{\mathrm{ex}^\infty{#1}}
\newcommand\sEx[1]{\str{\Ex{#1}}}
\newcommand\sex[1]{\str{\ex{#1}}}
\newcommand\VB[1]{\mathcal{P}_{#1}}
\newcommand\sVB[1]{\str{\mathcal{P}}_{#1}}

\swapnumbers
\theoremstyle{definition}
\newtheorem{defi}{Definition}[section]

\newtheorem{lemma}[defi]{Lemma}

\newtheorem{bem}[defi]{Remark}
\newtheorem{satz}[defi]{Proposition}
\newtheorem{thm}[defi]{Theorem}
\newtheorem{cor}[defi]{Corollary}

\input xy
\xyoption{all}

\title{Nonstandard Model Categories and Homotopy Theory}
\author{Lars Brünjes, Christian Serpé}
\date{\today}
\email{lbrunjes@gmx.de}
\address{
  Christian Serpé \\
  Westfälische Wilhelms-Universität Münster \\
  Mathematisches Institut \\
  Sonderforschungsbereich 478 ``Geometrische Strukturen in der Mathematik'' \\
  Hittorfstr. 27 \\
  D-48149 Münster \\
  Germany}
\email{serpe@uni-muenster.de}


\date{\today}

\begin{document}

\begin{abstract}
In order to apply nonstandard methods to questions of algebraic
geometry we continue our investigation from \cite{enlcat} and show how
important homotopical constructions behave under enlargements.
\end{abstract}

\maketitle



\section{Introduction}
Modern algebraic geometry makes heavy use of categorical
constructions. Therefore, in order to apply nonstandard methods to algebraic
geometry, in \cite{enlcat} we started to study how
such categorical constructions behave under enlargements. In the papers
\cite{nsetale}, \cite{enlsch} and \cite{EoSII} we applied the results to constructions in algebraic geometry.
In this paper we want to continue the investigation of
\cite{enlcat}. For reasons of why we consider the use of nonstandard
methods in algebraic geometry worthwhile, we refer the reader to the introduction of
\cite{enlcat}.

\vspace{\abstand}

Homotopical methods have become more and more important in algebraic
geometry. For instance, the higher $K$-groups of schemes are defined
as homotopy groups of certain simplicial sets, and the construction of Voevodsky's
$\A^1$-homotopy category makes heavy use of the concept
of Quillen model categories. Here we therefore want to investigate the
behaviour of these concept under enlargements.

\vspace{\abstand}

We start in section 2 with general (strict) n-categories and their
enlargements. Not surprisingly, it turns out that this notion behaves well
under enlargements.

After recalling some definitions, we study the enlargements of
model categories in section 3, and we show that the external and internal homotopy category
of an internal model category coincide.

In section 4 we consider simplicial sets. In general, the enlargement
of a topological space will not again be a topological space.
But we will see that at least a *simplicial set can be restricted to
become a usual simplicial set. We show that there is a morphism from the external
to the internal homotopy group of an internal simplicial set.

In section 5 we study the K-theory of exact categories and their
enlargements. Then, now formulated in terms of ultraproducts,
we construct a nontrivial morphism from the K-theory of an ultraproduct
to the ultraproduct of the K-theory.

\section{Internal $n$-Categories}

\bigskip

\noindent
We work in ZFC with the additional assumption that any set
is element of a universe (compare \cite{SGA4I}[Exp. I.0]). \\

Recall that for a natural number $n\in\Nn$,
a \emph{(small, strict) $n$-category} $\cA$
is given by the following data (compare \cite{leinster}):\\
\begin{itemize}
  \item
    A diagram of sets
    \[
      \xymatrix{
        {A_n} \ar@<1mm>[r]^s \ar@<-1mm>[r]_t &
        {A_{n-1}} \ar@<1mm>[r]^s \ar@<-1mm>[r]_t &
        {\ldots} \ar@<1mm>[r]^s \ar@<-1mm>[r]_t &
        {A_1} \ar@<1mm>[r]^s \ar@<-1mm>[r]_t &
        {A_0},
      }
    \]
    where $A_0$ is called the set of \emph{objects},
    $A_1$ the set of \emph{morphisms}
    and $A_m$ the set of \emph{$m$-cells}
    (for $0\leq m\leq n$),
    $s$ is called \emph{source} and $t$ is called \emph{target},\\
  \item
    maps $\id{}:A_m\rightarrow A_{m+1}$,
    $\alpha\mapsto \id{\alpha}$ for $m=0,\ldots,n-1$,
    called \emph{identity}, and \\
  \item
    maps $\circ_p:A_m\times_{A_p}A_m\longrightarrow A_m$,
    $(\alpha',\alpha)\mapsto\alpha'\circ_p\alpha$
    for all $0\leq p<m\leq n$
    (where the fibred product is taken with respect to
    $s^{m-p}$ and $t^{m-p}$),
    called \emph{composition},\\
\end{itemize}
subject to the following conditions:
\begin{itemize}
  \item
    $ss(\alpha)=st(\alpha)=ts(\alpha)=tt(\alpha)$ for $\alpha\in A_m$,
    $2\leq m\leq n$,\\
  \item
    \[
      \;\;\;\;\;\;\;\;\;\;
      s(\alpha'\circ_p\alpha)=
      \left\{\begin{array}{ll}
        s(\alpha) & ,m=p+1 \\
        s(\alpha')\circ_ps(\alpha) & ,m\geq p+2
      \end{array}\right.
      \;\;\mbox{and}\;\;\;\;
      t(\alpha'\circ_p\alpha)=
      \left\{\begin{array}{ll}
        t(\alpha) & ,m=p+1 \\
        t(\alpha')\circ_pt(\alpha) & ,m\geq p+2
      \end{array}\right.
    \]
    for $0\leq p<m\leq n$ and $\alpha,\alpha'\in A_m$,\\
  \item
    $s\id{\alpha}=\alpha=t\id{\alpha}$ for $\alpha\in A_m$, $0\leq m<n$,\\
  \item
    $\id{}^{m-p}t^{m-p}(\alpha)\circ_p\alpha=\alpha=\alpha\circ_p\id{}^{m-p}s^{m-p}(\alpha)$
    for $\alpha\in A_m$ and $0\leq p<m\leq n$,\\
  \item
    $(\alpha''\circ_p\alpha')\circ_p\alpha=\alpha''\circ_p(\alpha'\circ_p\alpha)$
    for $\alpha''$, $\alpha'$ and $\alpha$ in $A_m$ with
    $(\alpha'',\alpha')$ and $(\alpha',\alpha)$ in $A_m\times_{A_p}A_m$
    and $0\leq p<m\leq n$,\\
  \item
    $\id{\alpha'}\circ_p\id{\alpha}=\id{\alpha'\circ_p\alpha}$
    for $(\alpha',\alpha)\in A_m\times_{A_p}A_m$
    and $0\leq p<m<n$ and \\
  \item
    $(\beta'\circ_p\beta)\circ_q(\alpha'\circ_p\alpha)=(\beta'\circ_q\alpha')\circ_p(\beta\circ_q\alpha)$
    for $(\beta',\beta)$ and $(\alpha',\alpha)$ in $A_m\times_{A_p}A_m$
    with $(\beta',\alpha')$ and $(\beta,\alpha)$ in $A_m\times_{A_q}A_m$
    and $0\leq q<p<m\leq n$.\\
\end{itemize}

\noindent
Note that a 0-category is just a set
and that a 1-category is an ordinary category.\\

For $n\in\Nn$ and $n$-categories $\cC=\langle A_m,s,t,\id{},\circ_p\rangle$
and $\cD=\langle B_m,s,t,\id{},\circ_p\rangle$,
a \emph{(covariant) $n$-functor $F:\cC\longrightarrow\cD$}
is given by maps $A_m\longrightarrow B_m$, $\alpha\mapsto F\alpha$
for $0\leq m\leq n$ that commute with $s$, $t$, $\id{}$ and $\circ_p$ in the
obvious way.\\

It is easy to see that we get a category $\Cat{n}$ in this way
whose objects are all small $n$-categories and whose morphisms are
$n$-functors.

Note that $\Cat{0}=\Sets$ is just the category of sets,
and $\Cat{1}=\CAT$ is the category of categories.

\vspace{\abstand}

\begin{defi}\label{defucat}
  Let $\cU$ be a set, and let $n\in\Nn$ be a natural number.
  \begin{enumerate}
    \item\label{ucat}
        A \emph{$\cU$-small $n$-category} is an $n$-category
        $\langle A_m,s,t,\id{},\circ_p\rangle$
        with $A_m\in\cU$ for all $0\leq m\leq n$.
    \item\label{ncatu}
        $\UCat{\cU}{n}$ is the full subcategory of $\Cat{n}$
        with objects the $\cU$-small $n$-categories.\\
  \end{enumerate}
\end{defi}

Note that for $\cU$ a \emph{universe} and $n$ as above,
$\UCat{\cU}{n}$ is \emph{not} $\cU$-small. For example,
$\UCat{\cU}{0}$ is the category of $\cU$-sets, so the set of its objects
is $\cU\not\in\cU$.
On the other hand, if we choose a hierarchy $\cU_0\in\cU_1\in\cU_2\in\ldots$
of universes,
then $\UCat{\cU_n}{n}$ is $\cU_{n+1}$-small for all $n\in\Nn$.

Let $S$ be a set of ``individuals'' with cardinality
$|S|\geq|\bigcup_{n\in\Nn}\cU_n|$,
and let $*:\hat{S}\longrightarrow\widehat{\str{S}}$ be an enlargement.

Our discussion shows that the categories
$\UCat{\cU_0}{0}$, $\UCat{\cU_1}{1}$, $\UCat{\cU_2}{2}$, \ldots
are all $\hat{S}$-small (in the sense of \cite{enlcat} if we consider
the universes $\cU_i$ as subsets of $S$ and hence as elements of $S_1\setminus S$),
so we can consider their enlargements $\str{[\UCat{\cU_n}{n}]}$ in $\widehat{\str{S}}$
which are $\widehat{\str{S}}$-small categories.\\

\begin{defi}\label{defintcat}
  For $n\in\Nn$,
  an \emph{$n$-*category} is an object
  of $\sCat{n}:=\str{[\UCat{\cU_n}{n}]}$,
  and a (covariant) \emph{$n$-*functor}
  $F:\cC\longrightarrow\cD$
  between $n$-*categories $\cC$ and $\cD$
  is a morphism $\cC\longrightarrow\cD$ in $\sCat{n}$.\\
\end{defi}

\begin{satz}\label{satzintcat}
  $\sCat{n}$ is the category whose objects
  are $\str{\cU_n}$-small \emph{internal} $n$-categories
  and whose morphisms are \emph{internal} $n$-functors.
\end{satz}

\begin{proof}
  The objects of $\UCat{\cU_n}{n}$ are all tuples
  $\langle A_m,s,t,\id{},\circ_p\rangle$ with $A_m\in\cU_n$
  that satisfy the conditions stated above.
  Since these conditions are obviously first order,
  transfer proves the object part of the proposition.
  The morphism part is analogous but even simpler.\\
\end{proof}

In view of \ref{satzintcat},
we will from now on use the terms ``internal $n$-categories''
and ``internal $n$-functors'' for $n$-*categories and $n$-*functors
respectively (and drop ``$\str{\cU_n}$-small'' from the notation).\\

Note that \ref{satzintcat} shows in particular
that if $\cC$ is a $\cU_n$-small $n$-category,
then $\str{\cC}$ is also an $n$-category.
Also note that --- as $\UCat{\cU_n}{n}$ clearly has infinitely many
objects --- not all internal $n$-categories are of the form $\str{\cC}$
for an $n$-category $\cC$.

For example,
since $\UCat{\cU_1}{1}$ obviously contains the categories of
finite dimensional vector spaces over $\F_p$ for all finite primes $p$,
$\sCat{1}$ contains the internal categories of *finite dimensional
$\F_P$-vector spaces
for all \emph{infinite} primes $P\in\Nns$,
and these categories are clearly
not of the form $\str{\cC}$ for any
category $\cC$.\\

\section{Internal Model Categories}

\bigskip

\noindent
Recall the notion of \emph{model category} from \cite{quillen}: \\

A \emph{(small) model category} is a quadruple $\langle\cM,W,F,C\rangle$,
with a small category $\cM=\langle M_0,M_1,s,t,\id{},\circ\rangle$
and sets of morphisms $W,F,C\subseteq M_1$,
subject to the following conditions:\\
\begin{itemize}
  \item
    All \emph{finite} limits and colimits exist in $\cM$,\\
  \item
    $\forall X\in M_0:\ \id{X}\in W\cap F\cap C$,\\
  \item
    $\forall X\in\{W,F,C\}:\ \forall(f,g)\in X\times_{M_0}X:\ fg\in X$,\\
  \item
    $\forall(f,g)\in M_1\times_{M_0}M_1:\
    \bigl[fg\in W\wedge(f\in W\vee g\in W)\bigr]\Longrightarrow\{f,g\}\subseteq W$,\\
  \item
     $\forall X\in\{W,F,C\}:\
      \forall(r,i),(r',i')\in M_1\times_{M_0}M_1:\
      \forall f\in M_1:\
      \forall g\in X:$
     \begin{flushright}
       $\bigl[si=sf=tr\wedge si'=tf=tr'\wedge sg=ti\wedge tg=ti'
          \wedge ri=\id{si}\wedge r'i'=\id{si'}\bigr]
        \Longrightarrow
        f\in X$,
     \end{flushright}
  \item
    $\forall(p,f)\in F\times_{M_0}M_1:\
     \forall(g,i)\in M_1\times_{M_0}C:\
     \bigl[
       tp=tg\wedge
       si=sf\wedge
       (i\in W\vee p\in W)
     \bigr]
     \Longrightarrow$
     \begin{flushright}
       $\exists h\in M_1:\
        sh=ti\wedge
        th=sp\wedge
        hi=f\wedge
        ph=g$,
     \end{flushright}
  \item
    $\forall f\in M_1:\
     \bigl[
       \exists(p,i)\in(F\cap W)\times_{M_0}C:\ f=pi
     \bigr]
     \wedge
     \bigl[
       \exists(p,i)\in F\times_{M_0}(C\cap W):\ f=pi
     \bigr]$.\\
\end{itemize}

If $\cM$ is a category, then a triple $\langle W,F,C\rangle$
of classes of morphisms of $\cM$ is called a \emph{model structure on $\cM$}
if $\langle\cM,W,F,C\rangle$ is a model category.\\

By definition, a \emph{Quillen functor} between model categories
$\langle\cM,W,F,C\rangle$ and $\langle\cM',W',F',C'\rangle$
is an adjunction $\langle L,R,\varphi\rangle$ between $\cM$ and $\cM'$
satisfying\\

\begin{itemize}
  \item
    $\forall c\in C:\ Lc\in C'$ and\\
  \item
    $\forall f\in F':\ Rf\in F$.\\
\end{itemize}

A \emph{natural transformation} between Quillen functors
$\langle L,R,\varphi\rangle,\langle L',R',\varphi'\rangle:\cM\longrightarrow\cM'$
is just a natural transformation from $L$ to $L'$.\\

We chose to put the definition in this --- admittedly not very readable ---
way in order to show that everything
(except maybe the first condition)
is first order.
That also this first condition on
the existence of finite limits and colimits is first order
was explained in detail in \cite{enlcat}.\\

\begin{defi}\label{defintmodcat}
  Call a model category $\langle\cM,W,F,C\rangle$ \emph{$\cU_1$-small}
  if $\cM$, the underlying category, is $\cU_1$-small,
  and let $\MCAT^{\cU_1}$ be the 2-category with $\cU_1$-small model categories
  as objects,
  Quillen functors as morphisms
  and natural transformations between Quillen functors as 2-cells.

  $\MCAT^{\cU_1}$ is obviously $\cU_2$-small and hence an object
  of $\UCat{\cU_2}{2}$.
  Put $\sMCAT:=\str{[\MCAT^{\cU_1}]}$.
  This is an object of $\sCat{2}$ and hence by definition an internal
  2-category ---
  call its objects \emph{$^*$model categories},
  its morphisms \emph{*Quillen functors}
  and its 2-cells \emph{*natural transformations}.\\
\end{defi}

\begin{satz}\label{satzintmodcat}
  $\sMCAT$ is the 2-category having
  $\str{\cU_2}$-small \emph{internal} model categories as objects,
  \emph{internal} Quillen functors as morphisms
  and \emph{internal} natural transformations as 2-cells.
\end{satz}

\begin{proof}
  By transfer,
  an object of $\sMCAT$ is a quadruple
  $\langle\cM,W,F,C\rangle$ with $\cM$ a $\str{\cU_1}$-small category,
  subject to the transferred version of the conditions defining
  a model category.
  Since these conditions are first order,
  the object part of the proposition follows.
  The rest follows from \cite{enlcat}, where we show that adjunctions
  and natural transformations are first order.\\
\end{proof}

\begin{bem}\label{beminttwocat}
  Examples of 2-categories whose structure is defined by first order conditions
  abound and can be treated in a similar way:
  In \cite{enlcat} we have seen (although in a slightly different setup)
  that the 2-categories of additive, abelian, triangulated categories and posets
  as well as those of (additive, abelian, triangulated) fibrations
  all belong to that category.

  Another example is the 2-category of \emph{exact categories}
  that we will need for our treatment of algebraic $K$-theory below.

  In all these cases, the analogue of \ref{satzintmodcat} holds:
  *Additive categories are precisely the (small) internal additive categories,
  *abelian categories are precisely the (small) internal abelian categories,
  and so on.\\
\end{bem}

Because of \ref{satzintmodcat}, we will
refer to *model categories, *Quillen functors and *natural transformations
as internal model categories
(again dropping ``$\str{\cU_1}$-small'' from the notation),
internal Quillen functors and internal natural transformations
respectively.\\

Before we can formulate the next result, we have to recall some more definitions and facts from the theory of model categories.
Let $\langle\cM,W,F,C\rangle$ be a model category.\\

\begin{itemize}
  \item
    Since $\cM$ has all finite limits and colimits, it in particular has a terminal and an initial object,
    which we denote by $*$ and $\emptyset$ respectively.\\
  \item
    An object $X$ of $\cM$ is called \emph{cofibrant}
    if the unique morphism $\emptyset\longrightarrow X$
    is in $C$,
    and \emph{fibrant}
    if the unique morphism $X\longrightarrow *$
    is in $F$.\\
  \item
   For all objects $X$ in $\cM$,
   there is a morphism $QX\longrightarrow X$ in $F\cap W$
   with $QX$ cofibrant and a morphism $X\longrightarrow RX$ in $C\cap W$
   with $RX$ fibrant. (It follows from the axioms that $RQX$ is both
   fibrant and cofibrant.)\\
  \item
    Let $A$ and $X$ be cofibrant respectively fibrant objects of $\cM$,
    and let $f,g\in\Mor{\cM}{A}{X}$ be morphisms.
    Then $f$ and $g$ are called \emph{homotopic}, $f\sim g$,
    if there exists a commutative diagram
    
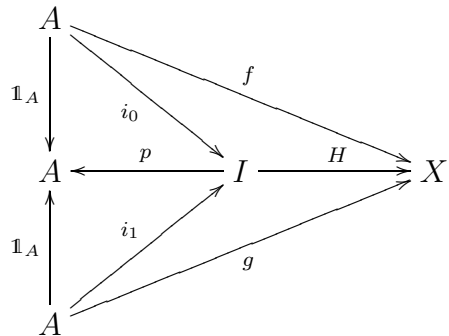
\[
      \xymatrix@C=20mm@R=15mm{
        A \ar[dr]_{i_0} \ar[d]_{\id{A}} \ar[drr]^f \\
        A &
        {I} \ar[r]^H \ar[l]_p &
        X \\
        A \ar[ur]^{i_1} \ar[u]^{\id{A}} \ar[urr]_g \\
      }
    \]
    in $\cM$ with $p\in W$. Homotopy is an equivalence relation,
    and the set of equivalence (or homotopy) classes
    $\Mor{\cM}{A}{X}/\sim$
    is denoted by $\pi(A,X)$.\\
\end{itemize}

By transfer, we get corresponding notions of *(co-)fibrant objects
and *homotopic morphisms for internal model categories.\\

\begin{satz}\label{satzhomotopic}
  Let $\cM=\langle\cM,W,F,C\rangle$ be an internal model category,
  let $A$ and $X$ be objects of $\cM$
  with $A$ *cofibrant and $X$ *fibrant, and let $f,g:A\longrightarrow X$
  be morphisms in $\cM$.

  Then $A$ is cofibrant, and $X$ is fibrant,
  and $f$ and $g$ are *homotopic
  if and only if they are homotopic.
  In particular, the set of *homotopy classes equals $\pi(A,X)$,
  the set of homotopy classes.
\end{satz}

\begin{proof}
  This is easy, because
  the *terminal object of $\cM$ is the terminal object,
  the *initial object is the initial object,
  and the condition defining homotopy is obviously first order.\\
\end{proof}

By definition, the \emph{homotopy category} $\Ho{\cM}$
of a model category $\langle\cM,W,F,C\rangle$
is the localization $\cM[W^{-1}]$ of $\cM$ by the set $W$.
By transfer, we get the notion of the \emph{*homotopy category}
$\sHo{\cM}$ of an internal category $\langle\cM,F,C,W\rangle$,
which is the *localization of $\cM$ by $W$.\\

Note that for an arbitrary internal category $\cC$,
it will not be true that its *localization
by an (internal) set $W$ of morphisms
agrees with its localization by $W$, because in general
the morphisms in the localization are \emph{finite} words
built from morphisms in $\cC$ and formal inverses of morphisms in $W$,
whereas words in the *localization can be of *finite length and
do not have to be finite.

Nevertheless, the situation is much better with model categories,
as the next result shows:\\

\begin{thm}\label{thmmodel}
  Let $\langle\cM,F,C,W\rangle$ be an internal model category.
  Then $\Ho{\cM}$ and $\sHo{\cM}$ are canonically isomorphic.
\end{thm}

\begin{proof}
  By the universal property of localization,
  the canonical functor $\cM\longrightarrow\sHo{\cM}$ factors
  uniquely as $\cM\longrightarrow\Ho{\cM}\xrightarrow{F}\sHo{\cM}$,
  where $F$ is the identity on objects.
  We claim that $F$ is also fully faithful.
  For this, let $A$ and $X$ be arbitrary objects of $\cM$.
  The diagrams $RQA\longleftarrow QA\longrightarrow A$
  and $RQX\longleftarrow QX\longrightarrow X$ give isomorphisms
  in $\Ho{\cM}$ and $\sHo{\cM}$, so by composing with them
  we can assume without loss of generality that $A$ and $X$
  are both fibrant and cofibrant.
  It is a well known fact that in this case $\Mor{\Ho{\cM}}{A}{X}=\pi(A,X)$,
  and the theorem follows from this fact, its transferred version
  and \ref{satzhomotopic}.\\
\end{proof}

\section{Internal Simplicial Sets}

\bigskip

Let $\Delta$ be the \emph{simplicial category},
i.e. the category whose objects are the finite ordinals
$[n]=\{0,1,\ldots,n\}$ for $n\in\Nn$
and whose morphisms are monotonic maps;
this is of course a $\cU_1$-small category, and it follows
immediately by transfer that $\str{\Delta}$ has objects
$[n]$ for $n\in\Nns$ and again monotonic maps as morphisms.

The category $\SSets$ of \emph{simplicial sets}
is by definition the category of functors $X_\bullet:\op{\Delta}\longrightarrow\Sets$,
$[n]\mapsto X_n$,
and for a set $\cU$, we denote by $\USSets{\cU}$ the full subcategory
of $\SSets$ of functors $X_\bullet$ with $X_n\in\cU$ for all $n\in\Nn$.

$\USSets{\cU_0}$ is obviously $\cU_1$-small,
and we put $\sSSets:=\str{\USSets{\cU_0}}$ and call the resulting *category
the category of *simplicial sets.

It is clear that $\sSSets$ is the category of internal functors $X_\bullet$
from $\str{\op{\Delta}}$ to the category of $\str{\cU_0}$-sets
with internal maps as morphisms,
and we use the term \emph{internal simplicial set} as a synonym for
``*simplicial set''.\\

\begin{defi}\label{defres}
  The functor
  $*:\Delta\longrightarrow\str{\Delta}$
  (which is the identity on objects and morphisms and obviously a
  full embedding)
  induces by composition a canonical ``restriction'' functor
  $\sSSets\longrightarrow\SSets$,
  which ``cuts off'' the infinite part of an internal simplicial set.
  We denote this functor by $\res$.\\
\end{defi}

Recall that there is an adjunction
$\langle|.|,\Sing,\varphi\rangle$ from $\SSets$ to $\Top$,
the category of topological spaces, in terms of which $\SSets$ can
be given a model structure by setting\\
\begin{itemize}
  \item
    $W:=\bigl\{w:X_\bullet\longrightarrow Y_\bullet\
    \bigl\vert\
    \mbox{$|w|:|X_\bullet|\longrightarrow|Y_\bullet|$
    is a weak homotopy equivalence}
    \bigr.\bigr\}$,\\
  \item
    $F:=\Bigl\{f:X_\bullet\longrightarrow Y_\bullet\
    \Bigl\vert\
    \forall n\in\Np:\
    \forall 0\leq k\leq n:\
    \forall x_0,\ldots,x_{k-1},x_{k+1},\ldots,x_n\in X_{n-1}:\
    \forall y\in Y_n:$\\
    $\mbox{}\hspace{13mm}
    \Bigl(\bigl[\forall 0\leq i<j\leq n:\
      i\neq k\neq j\Longrightarrow
      \delta^{j-1}x_i=\delta^ix_j
    \bigr]\wedge
    \bigl[
      \forall i\in[n]\setminus\{k\}: \delta^iy=fx_i
    \bigr]\Bigr)\Longrightarrow$
    \begin{flushright}
      $
      \exists x\in X_n:\
      px=y\wedge
      \forall i\in[n]\setminus\{k\}:\
      \delta^ix=x_i
      \Bigr.\Bigr\}$,
    \end{flushright}
  \item
    $C:=\bigl\{c:X_\bullet\longrightarrow Y_\bullet\
    \bigl\vert\
    \mbox{$c_n:X_n\longrightarrow Y_n$ is injective for all $n\in\Nn$}
    \bigr.\bigr\}$.\\

\end{itemize}
Furthermore,
there is a model structure on $\Top$ that turns
$\langle|.|,\Sing,\varphi\rangle:\SSets\longrightarrow\Top$
into a Quillen functor (which induces an equivalence on the
associated homotopy categories).

By restricting to $\USSets{\cU_0}$ and $\Top^{\cU_0}$ (the category
of topological spaces whose underlying sets are $\cU_0$-small),
we get a model structure on $\USSets{\cU_0}$ and a Quillen functor
$\USSets{\cU_0}\longrightarrow\Top^{\cU_0}$
and hence an internal model structure on $\sSSets$ and an internal
Quillen functor $\sSSets\longrightarrow\str{\Top^{\cU_0}}$.\\

For $n\in\Nn$ and $i\in[n+1]$, $d^i:[n]\longrightarrow[n+1]$
denotes the unique injective monotonic map
whose image is $[n+1]\setminus\{i\}$,
and for $n\in\Nn$ and $i\in[n]$,
$s^i:[n+1]\longrightarrow[n]$ denotes the unique surjective monotonic map
with $s^i(i)=s^i(i+1)$,
for a simplicial set $X_\bullet$,
$\delta_i:X_{n+1}\longrightarrow X_n$ denotes $X(d^i)$,
and $\sigma_i:X_n\longrightarrow X_{n+1}$ denotes $X(s^i)$.\\

\begin{satz}\label{satzrescompatible}
  The functor $\res:\sSSets\longrightarrow\SSets$
  respects $F$, $C$ and
  the terminal, the initial and fibrant and cofibrant objects.
\end{satz}

\begin{proof}
  Since $F$ and $C$ for $\SSets$ are defined by first order formulas,
  the same formulas define $F$ and $C$ for $\sSSets$
  (except that $\Np$ and $\Nn$ have to be replaced by $\Nps$ respectively $\Nns$).
  This shows that $\res$ respects $F$ and $C$.

  The terminal object of $\SSets$ is the simplicial set
  $*_\bullet$ with $*_n=\{*\}$ for all $n\in\Nn$,
  so the terminal object of $\sSSets$ is the internal simplicial set
  $\str{*}_\bullet$ with $\str{*}_n=\{*\}$ for all $n\in\Nns$,
  which is clearly mapped to $*_\bullet$ under $\res$.
  Similarly, the initial object of $\SSets$ is the constant
  simplicial set $\emptyset_\bullet$ with $\emptyset_n=\emptyset$
  for all $n\in\Nn$,
  the initial internal simplicial set is $\str{\emptyset}_\bullet$
  with $\str{\emptyset}_n=\emptyset$ for all $n\in\Nns$,
  and again $\res\ \str{\emptyset}_\bullet=\emptyset_\bullet$
  holds trivially.

  Thus by definition of ``fibrant'' and ``cofibrant'',
  the rest of the proposition follows as well.\\
\end{proof}

\noindent
Recall the following definitions and results from \cite{kan2}
for simplicial sets:\\
\begin{itemize}
  \item
    If $X_\bullet$ is fibrant and $n\in\Nn$,
    then there is an equivalence relation $\sim$ on $X_n$
    defined by putting $y\sim z$ if and only if
    \[
      \left\{\begin{array}{ll}
        \exists x\in X_1:\ \delta^0x=y\wedge\delta^1x=z & ,n=0 \\
        \bigl[\forall 0\leq i\leq n:\ \delta^iy=\delta^iz\bigr]\wedge
        \exists x\in X_{n+1}:\
        \delta^ix=\left\{\begin{array}{ll}
          \sigma^{n-1}\delta^iy & ,0\leq i\leq n \\
          y & ,i=n \\
          z & ,i=n+1 \\
        \end{array}\right\} & ,n\geq 1\\
      \end{array}\right.
    \]
    \mbox{}
  \item
    A \emph{simplicial set with base point} is a pair $\langle X_\bullet,x\rangle$,
    where $X_\bullet$ is a simplicial set and $x\in X_0$.
    A \emph{morphism} $\langle X_\bullet,x\rangle\longrightarrow\langle Y_\bullet,y\rangle$
    of simplicial sets with base points is a morphism $F:X_\bullet\longrightarrow Y_\bullet$
    of simplicial sets with $f_0(x)=y$.
    We denote the resulting category by $\SSetsb$.\\
  \item
    For a fibrant simplicial set with base point $\langle X_\bullet,x\rangle$
    and $n\in\Nn$, define the \emph{$n$-th simplicial homotopy set of $\langle X_\bullet,x\rangle$}
    as $\tilde{\pi}_n(X_\bullet,x):=I_n/\sim$, where
    \[
      X_n\supseteq I_n:=
      \left\{\begin{array}{ll}
        X_0 & ,n=0 \\
        \bigl\{y\in X_1\
        \bigl\vert\ \delta^0y=\delta^1y=x
        \bigr.\bigr\} & ,n=1 \\
        \bigl\{y\in X_n\
        \bigl\vert\ \forall 0\leq i\leq n:\
          \delta^iy=\sigma^{n-2}\ldots\sigma^0x
        \bigr.\bigr\} & ,n\geq 2. \\
      \end{array}\right.
    \]
    The set $\tilde{\pi}_0(X_\bullet,x)$ is pointed by $\bar{x}$,
    and for $n\geq 1$, $\tilde{\pi}_n(X_\bullet,x)$ is a group,
    the \emph{$n$-th simplicial homotopy group of $\langle X,x\rangle$},
    whose multiplication is given as follows:
    Let $y,y'\in I_n$.
    Since $X_\bullet$ is fibrant,
    there exists $z\in X_{n+1}$ satisfying
    \[
      \delta^{n-1}z=y\wedge
      \delta^{n+1}z=y'\wedge
      \forall 0\leq i\leq n-2:\
        \delta^iz=\sigma^{n-1}\ldots\sigma^0x.
    \]
    Put
    $\overline{y}\cdot\overline{y'}:=\overline{\delta^nz}\in\tilde{\pi}_n(X,x)$.\\
  \item
    If $f_\bullet:\langle X_\bullet,x\rangle\longrightarrow\langle Y_\bullet,y\rangle$
    is a morphism of \emph{fibrant} simplicial sets with base points,
    then $f_n:X_n\longrightarrow Y_n$ induces a map
    $\tilde{\pi}_n(f):\tilde{\pi}_n(X_\bullet,x)\longrightarrow\tilde{\pi}_n(Y_\bullet,y)$
    for all $n\in\Nn$ which is a morphism of pointed sets for $n=0$
    and a group homomorphism for $n\geq 1$.
    In this way, we get functors $\tilde{\pi}_n$ from the full subcategory of
    $\SSetsb$ of fibrant objects to the category of pointed sets (for $n=0$)
    respectively the category of groups (for $n\geq 1$).\\
  \item
    If $\langle X_\bullet,x\rangle$ is an arbitrary object of $\SSetsb$,
    the \emph{$n$-th homotopy set}
    respectively \emph{group} is defined as
    $\pi_n(X_\bullet,x):=\pi_n(|X_\bullet|,|x|)$ for $n\in\Nn$, so
    $\pi_n$ defines a functor from $\SSetsb$ to
    the category of pointed sets respectively groups.
    This functor --- when restricted to the full subcategory of fibrant
    objects --- is canonically isomorphic
    to the functor $\tilde{\pi}_n$ for all $n\in\Nn$.\\
  \item
    A sequence of morphisms in $\SSetsb$
    \[
      \xymatrix{
        {\langle F_\bullet,f\rangle} \ar[r]^{q_\bullet} &
        {\langle E_\bullet,e\rangle} \ar[r]^{p_\bullet} &
        {\langle B_\bullet,b\rangle} \\
      }
    \]
    is called a \emph{fibre sequence} if
    \begin{itemize}
      \item
        $p\in F\wedge\forall n\in\Nn:\
        \mbox{$p_n$ is surjective}$,
      \item
        $q\in C$ and
      \item
        $
          \forall n\in\Nn:\
          q_n(F_n)=
          \left\{\begin{array}{ll}
            p_0^{-1}(b) & ,n=0 \\
            p_n^{-1}(\sigma^{n-1}\ldots\sigma^0b) & ,n\geq 1 \\
          \end{array}\right.
        $ \\
    \end{itemize}
    A fibre sequence as above functorially defines a long exact
    sequence
    \[
      \ldots \xrightarrow{\pi_{n+1}(p_\bullet)}
      \pi_{n+1}(B_\bullet,b)\xrightarrow{\delta}
      \pi_n(F_\bullet,f)\xrightarrow{\pi_n(q_\bullet)}
      \pi_n(E_\bullet,e)\xrightarrow{\pi_n(p_\bullet)}
      \pi_n(B_\bullet,b)\xrightarrow{\delta}
      \ldots
    \]
    of groups respectively pointed sets.\\
  \item
    There is an endofunctor $\Ex{}$ of $\SSets$
    and a natural transformation $\ex{}:\id{\SSets}\longrightarrow\Ex{}$
    with the following properties:
    \begin{itemize}
      \item
        $\Ex{X_\bullet}$ is fibrant for all simplicial sets $X_\bullet$,
      \item
        $\ex{}:X_\bullet\longrightarrow\Ex{X_\bullet}$
        is in $W$ for all simplicial sets $X_\bullet$.
      \item
        If
        \[
          \xymatrix{
            {\langle F_\bullet,f\rangle} \ar[r]^{q_\bullet} &
            {\langle E_\bullet,e\rangle} \ar[r]^{p_\bullet} &
            {\langle B_\bullet,b\rangle} \\
          }
        \]
        is a fibre sequence,
        then
        \[
          \xymatrix{
            {\langle\Ex{F_\bullet},\ex{(f)}\rangle} \ar[r]^{\Ex{q_\bullet}} &
            {\langle\Ex{E_\bullet},\ex{(e)}\rangle} \ar[r]^{\Ex{p_\bullet}} &
            {\langle\Ex{B_\bullet},\ex{(b)}\rangle} \\
          }
        \]
        is also a fibre sequence.\\
    \end{itemize}
\end{itemize}

These notions can be restricted to $\USSets{\cU_0}$
and then enlarged to $\sSSets$, i.e. we have
the category $\sSSetsb$ of internal simplicial sets with base points,
internal (simplicial) homotopy functors
$\str{\pi}_n$ and $\str{\tilde{\pi}_n}$ (for $n\in\Nns$)
and internal fibred sequences
as well as $\sEx{}$ and $\sex{}$
with analogous properties.

Note that $\res:\sSSets\longrightarrow\SSets$
induces a functor $\res:\sSSetsb\longrightarrow\SSetsb$
by sending the internal base point to itself.\\

\begin{thm}\label{thmpi}
  There is a canonical natural transformation
  $\varphi_n:\pi_n\circ\res\Longrightarrow\str{\pi}_n$
  of functors from $\sSSetsb$
  to groups respectively sets for all $n\in\Nn$.
  Moreover, if
  \[
    \xymatrix{
      {\langle F_\bullet,f\rangle} \ar[r]^{q_\bullet} &
      {\langle E_\bullet,e\rangle} \ar[r]^{p_\bullet} &
      {\langle B_\bullet,b\rangle} \\
    }
  \]
  is an internal fibre sequence,
  then
  \[
    \xymatrix{
      {\langle\res\ F_\bullet,f\rangle} \ar[r]^{\res\ q_\bullet} &
      {\langle\res\ E_\bullet,e\rangle} \ar[r]^{\res\ p_\bullet} &
      {\langle\res\ B_\bullet,b\rangle} \\
    }
  \]
  is a fibre sequence and
  \[
    \xymatrix{
      {\pi_{n+1}(\res\ B_\bullet,b)} \ar[r]^{\str{\delta}} \ar[d]_{\varphi_{n+1}} &
      {\pi_n(\res\ F_\bullet,f)} \ar[d]^{\varphi_n} \\
      {\str{\pi}_n(B_\bullet,b)} \ar[r]_\delta &
      {\str{\pi}_n(F_\bullet,f)} \\
    }
  \]
  commutes for all $n\in\Nn$.
\end{thm}

\begin{proof}
  Let $\langle X_\bullet,x\rangle$ be an object of $\sSSetsb$.
  Then on the one hand, we have canonical isomorphisms
  \[
    \xymatrix@C=17mm{
      {\alpha_n:\ \str{\pi_n}(X_\bullet,x)} \ar[r]^-{\str{\pi}_n(\str{\ex{}})}_-\sim &
      {\str{\pi_n}(\str{\Ex{X_\bullet}},\str{\ex{x}})} \ar[r]_-\sim &
      {\str{\tilde{\pi}_n}(\str{\Ex{X_\bullet}},\str{\ex{x}})}
    }
  \]
  for all $n\in\Nns$,
  on the other hand $\res\ \Ex{X_\bullet}$ is fibrant by \ref{satzrescompatible},
  so we get morphisms
  \[
    \xymatrix@C=17mm{
      {\beta_n:\ \pi_n(\res\ X_\bullet,x)} \ar[r]^-{\pi_n(\res\str{\ex{}})} &
      {\pi_n(\res\ \str{\Ex{X_\bullet}},\str{\ex{x}})} \ar[r]_-\sim &
      {\tilde{\pi}_n(\res\ \str{\Ex{X_\bullet}},\str{\ex{x}})}
    }
  \]
  for all $n\in\Nn$.
  But the definition of the $n$-th \emph{simplicial} homotopy set respectively
  group of a simplicial set with base point $\langle Y_\bullet,y\rangle$
  is obviously first order and only depends on
  the restriction of $Y$ to the full subcategory of $\Delta$ with objects
  $\{[0],[1],\ldots,[n+1]\}$, so that
  \[
    \tilde{\pi}_n(\res\ \str{\Ex{X_\bullet}},\str{\ex{x}})
    =
    \str{\tilde{\pi}_n}(\str{\Ex{X_\bullet}},\str{\ex{x}})
  \]
  for $n\in\Nn$.
  Putting everything together, we can define
  $\varphi_n:=\alpha_n^{-1}\beta_n$,
  and this is clearly functorial and hence defines a natural transformation.

  That $\res$ maps internal fibre sequences to fibre sequences
  is clear from \ref{satzrescompatible} and the definition of
  ``fibre sequence''.
  So the last part of the theorem follows from the fact
  that $\Ex{}$ respects fibre sequences and from the construction of $\varphi_n$.\\
\end{proof}

\section{Nonstandard $K$-Theory}

\bigskip

Recall the notion of an \emph{exact category} from \cite{quillenK}: \\

Let $\cM$ be a small additive category,
and let $\cE$ be a set of composable pairs
$\langle M'\xrightarrow{i}M,M\xrightarrow{j}M''\rangle$
of morphisms in $\cM$.
A morphism $i$ in $\cM$
is called an \emph{admissible monomorphism}
if there is a morphism $j$ in $\cM$
with $\langle i,j\rangle\in\cE$,
and $j\in\MorC{\cM}$
is called an \emph{admissible epimorphism}
if there exists $i\in\MorC{\cM}$ with $\langle i,j\rangle\in\cE$.

The pair $\langle\cM,\cE\rangle$ is a \emph{(small) exact category}
if the following conditions are satisfied:\\

\begin{itemize}
  \item
    If
    \[
      \xymatrix{
        {M'} \ar[r]^i \ar[d]_\alpha^\wr &
        M \ar[r]^j \ar[d]_\beta^\wr &
        {M''} \ar[d]_\gamma^\wr \\
        {N'} \ar[r]_{\tilde{i}} &
        N \ar[r]_{\tilde{j}} &
        {N''} \\
      }
    \]
    is a commutative diagram in $\cM$ with $\langle i,j\rangle\in\cE$,
    then $\langle\tilde{i},\tilde{j}\rangle\in\cE$,\\
  \item
    for all objects $M'$, $M''$ of $\cM$,
    \[
      \left\langle
        M'\xrightarrow{(\id{M'},0)}M'\oplus M'',
        M'\oplus M''\xrightarrow{\mathrm{pr}_2}M''
      \right\rangle\in\cE,
    \]
  \item
    for all $\langle i,j\rangle\in\cE$,
    $i$ is a kernel of $j$, and $j$ is a cokernel of $i$,\\
  \item
    the set of admissible epimorphisms is closed under composition,\\
  \item
    the set of admissible monomorphisms is closed under composition,\\
  \item
    pullbacks of admissible epimorphisms along arbitrary morphisms
    exist in $\cM$ and are again admissible epimorphisms,\\
  \item
    pushouts of admissible monomorphisms along arbitrary morphisms
    exist in $\cM$ and are again admissible monomorphisms,\\
  \item
    Let $M\xrightarrow{j}M''$ be a morphism in $\cM$ which
    has a kernel. If there is a morphism $N\xrightarrow{\beta}M$
    in $\cM$ such that $j\beta$ is an admissible epimorphism,
    then $j$ is an admissible epimorphism,\\
  \item
    Let $M'\xrightarrow{i}M$ be a morphism in $\cM$ which
    has a cokernel. If there is a morphism $M\xrightarrow{\beta}N$
    in $\cM$ such that $\beta i$ is an admissible monomorphism,
    then $i$ is an admissible monomorphism.\\
\end{itemize}

Recall further that an \emph{exact functor} between
exact categories $\langle\cM,\cE\rangle$ and $\langle\cM',\cE'\rangle$
is an additive functor $F:\cM\longrightarrow\cM'$ with
\[
  \forall\langle i,j\rangle\in\cE:\
  \langle Fi,Fj\rangle\in\cE',
\]
and a \emph{natural transformation} between exact functors
is just a natural transformation in the ordinary sense.

Note that if $\cA$ is an \emph{abelian} category,
then $\langle\cA,\cE\rangle$ is exact with
\[
  \cE:=\left\{(i,j)\
  \left\vert\
  \mbox{$0\longrightarrow A'\xrightarrow{i}A\xrightarrow{j}B''\longrightarrow 0$
  exact}
  \right.\right\},
\]
and if not stated otherwise, an abelian category is always considered
as an exact category in this way.

\begin{defi}\label{defintexcat}
  Call an exact category $\langle\cM,\cE\rangle$ \emph{$\cU_1$-small}
  if $\cM$, the underlying additive category, is $\cU_1$-small,
  and let $\ECAT^{\cU_1}$ be the 2-category
  with $\cU_1$-small exact categories as objects,
  exact functors as morphisms
  and natural transformations as 2-cells.

  $\ECAT^{\cU_1}$ is obviously $\cU_2$-small and hence an object
  of $\UCat{\cU_2}{2}$.
  Put $\sECAT:=\str{[\ECAT^{\cU_1}]}$.
  This is an object of $\sCat{2}$ and hence by definition an internal
  2-category ---
  call its objects \emph{$^*$exact categories},
  its morphisms \emph{*exact functors}
  and its 2-cells \emph{*natural transformations}.\\
\end{defi}

\begin{satz}\label{satzintexcat}
  $\sECAT$ is the 2-category having
  $\cU_2$-small \emph{internal} exact categories as objects,
  \emph{internal} exact functors as morphisms
  and \emph{internal} natural transformations as 2-cells.
\end{satz}

\begin{proof}
  Since the conditions defining an exact category are all first order,
  this is completely analogous to the proof of \ref{satzintmodcat}.
\end{proof}

By definition, the \emph{Quillen category} $\Qu{\langle\cM,\cE\rangle}$
(also simply denoted by $\Qu{\cM}$ if the set $\cE$ is understood)
associated to an exact
category $\langle\cM,\cE\rangle$
is the categories whose objects are the objects of $\cM$
and whose morphisms from $M$ to $M'$ are isomorphism classes of diagrams
\[
  (j,i):\ M\xleftarrow{j}N\xrightarrow{i}M'
\]
in $\cM$
with an admissible epimorphism $j$ and an admissible monomorphism $i$,
where two diagrams $(j,i)$ and $(\tilde{j},\tilde{i})$ are isomorphic
if there is a commutative diagram
\[
  \xymatrix@C=10mm@R=2mm{
    & {\tilde{N}} \ar[dl]_{\tilde{j}} \ar[dr]^{\tilde{i}}
      \ar[dd]^\beta_\wr \\
    M & & {M'} \\
    & N \ar[ul]^j \ar[ur]_i \\
  }
\]
in $\cM$ with an isomorphism $\beta$.
Composition of morphisms
$M\xleftarrow{j}N\xrightarrow{i}M'$
and $M'\xleftarrow{j'}N'\xrightarrow{i'}M''$
is defined as $(jj'',i'i''):M\longrightarrow M''$,
where $j'':M''\longrightarrow N$ is a pullback of $j'$ along $i$
and $i'':M''\longrightarrow N'$ is the corresponding projection
\[
  \xymatrix{
    {M''} \ar[d]_{j''} \ar[r]^{i''} \ar@{}[rd]|{\Box} &
    {N'} \ar[r]^{i'} \ar[d]^{j'} &
    {M''} \\
    N \ar[d]_j \ar[r]_i &
    {M'} \\
    {M} \\
  }
\]
(it follows from the axioms for an exact category that $i''$ and $j''$
are an admissible monomorphism respectively epimorphism).\\

In this way, we get the \emph{Quillen} 2-functor $\Qu{}:\ECAT\longrightarrow\Cat{1}$,
$\langle\cM,\cE\rangle\mapsto\Qu{\langle\cM,\cE\rangle}$,
and by transfer the \emph{*Quillen} 2-functor $\sQu{}:\sECAT\longrightarrow\sCat{1}$
from internal exact categories to internal categories.

\begin{lemma}\label{lemmaQuillen}
  Let $\langle\cM,\cE\rangle$ be an internal exact category.
  Then $\Qu{\cM}$ and $\sQu{\cM}$ are canonically isomorphic.
\end{lemma}

\begin{proof}
  This follows immediately from the fact that the definition of morphisms
  and their composition in $\Qu{\cM}$ is obviously first order.
\end{proof}

\noindent
Recall that the \emph{nerve} of a small category $\cC$
is the simplicial set $\Ne{\cC}_\bullet$,
where $\Ne{\cC}_n$ is the set of sequences
\[
  X_0\xrightarrow{f_0}X_1\xrightarrow{f_1}\ldots\xrightarrow{f_{n-1}}X_n
\]
in $\cC$
and where $\delta_i$ (respectively $\sigma_i$) is obtained
by deleting the object $X_i$ (respectively replacing $X_i$ by
$X_i\xrightarrow{\id{X_i}}X_i$) in the evident way.
If $F:\cC\longrightarrow \cC'$ is a functor,
we obviously get an induced morphism of simplicial sets
$\Ne{F}:\Ne{\cC}\longrightarrow\Ne{\cC'}$ between the nerves.

By transfer, we define the notion of the \emph{*nerve} $\sNe{\cC}$
of an internal category $\cC$, which is a *simplicial set.

\begin{lemma}\label{lemmanerve}
  If $\cC$ is an internal category,
  then $\res(\sNe{\cC})=\Ne{\cC}$,
  and if $F:\cC\longrightarrow\cD$ is an internal functor between internal categories,
  then $\res(\sNe{F})=\Ne{F}$.
\end{lemma}

\begin{proof}
  This is obvious from the definition of the nerve.
\end{proof}

\noindent
Let $\langle\cM,\cE\rangle$ be a small exact category,
and let $n\in\Nn$. Then the \emph{$n$-th $\mathrm{K}$-group}
of $\cM$
is defined as
\[
  \K{n}{\cM}:=\pi_{n+1}\bigl(\Ne{\Qu{\cM}},0\bigr)
\]
(this is abelian even for $n=0$),
and by transfer, for $n\in\Nns$,
the \emph{$n$-th $\str{\mathrm{K}}$-group}
of an internal exact category $\langle\cM,\cE\rangle$ is
\[
  \sK{n}{\cM}:=\str{\pi}_{n+1}\bigl(\sNe{\sQu{\cM}},0\bigr).
\]

Let $\cA$ be a (small) abelian category,
and let $\cB\subseteq\cA$ be a \emph{Serre subcategory},
i.e. a full abelian subcategory such that for every short exact sequence
\[
  0\longrightarrow A'\longrightarrow A\longrightarrow A''\longrightarrow 0
\]
in $\cA$, $A$ is in $\cB$ if and only if $A'$ and $A''$ are both in $\cB$.
Recall that in this situation, $\cA/\cB$ is the abelian category whose objects are the objects of $\cA$
and whose morphisms from $A$ to $A'$ are represented by diagrams in $\cA$ of the form
\[
  (\sigma,\varphi):\ A\xleftarrow{\sigma}C\xrightarrow{\varphi}A'
\]
with kernel and cokernel of $\sigma$ in $\cB$,
where another diagram
\[
  (\sigma',\varphi'):\ A\xleftarrow{\sigma'}C'\xrightarrow{\varphi'}A'
\]
represents the same morphism if and only if there is a commutative diagram
\[
  \xymatrix{
    & C \ar[dl]_\sigma \ar[dr]^\varphi \\
    A & {C''} \ar[l]_{\sigma''} \ar[r]^{\varphi''} \ar[u]^\tau \ar[d]_{\tau'} & {A'} \\
    & {C'} \ar[ul]^{\sigma'} \ar[ur]_{\varphi'} \\
  }
\]
in $\cA$ where $\tau$, $\tau'$ and $\sigma''$ have kernel and cokernel in $\cB$.
The composition of two morphisms $[A\xleftarrow{\sigma}C\xrightarrow{\varphi}A']$
and $[A'\xleftarrow{\sigma'}C'\xrightarrow{\varphi'}A'']$
is represented by $(\sigma\sigma'',\varphi'\varphi'')$,
where $\sigma''$ is the pullback of $\sigma'$ along $\varphi$ and $\varphi''$
is the corresponding projection:
\[
  \xymatrix{
    A & C \ar@{}[rd]|\Box \ar[l]_\sigma \ar[d]_\varphi & {C''} \ar[l]_{\sigma''} \ar[d]^{\varphi''} \\
    & {A'} & {C'} \ar[l]^{\sigma'} \ar[d]^{\varphi'} \\
    & & {A''}. \\
  }
\]
(One can easily check that this is well defined.)

The canonical exact functor $F:\cA\longrightarrow\cA/\cB$ has the property
that an object $B$ of $\cA$ is in $\cB$ if and only if $FB$ is isomorphic to $0$ in $\cA/\cB$,
and if $F':\cA\longrightarrow\cA'$ is an exact functor
between abelian categories
satisfying $F'B\cong 0$
for all objects $B$ of $\cB$, then $F'$ factors uniquely over $F$.

By transfer,
the *quotient $\cA\ \str{/}\ \cB$ of a *abelian category $\cA$ by a
\emph{*Serre subcategory} $\cB$ is also defined.

\begin{lemma}\label{lemmaSerre}
  Let $\cB$ be a *Serre subcategory of an internal abelian category $\cA$.
  Then $\cB$ is a Serre subcategory of $\cA$, and the abelian categories
  $\cA/\cB$ and $\cA\ \str{/}\ \cB$ are canonically isomorphic.
\end{lemma}

\begin{proof}
  This is again obvious, because both the definitions of "Serre subcategory"
  and of "quotient by a Serre subcategory" are obviously first order.
\end{proof}

The \emph{localization theorem} in (higher) algebraic K-theory states
that if $\cB$ is a Serre subcategory of an abelian category $\cA$,
then
\[
  \Ne{\Qu{\cB}}_\bullet\longrightarrow
  \Ne{\Qu{\cA}}_\bullet\longrightarrow
  \Ne{\Qu{(\cA/\cB)}}_\bullet
\]
is a fibre sequence and hence induces a long exact sequence of K-groups
\[
  \ldots\longrightarrow
  \K{n+1}{\cA/\cB}\xrightarrow{\delta}
  \K{n}{\cB}\longrightarrow
  \K{n}{\cA}\longrightarrow
  \K{n}{\cA/\cB}\longrightarrow\ldots.
\]
Furthermore, if $\cB'$ is a Serre subcategory of $\cA'$
and if $F:\cA\longrightarrow\cA'$ is an exact functor
that maps $\cB$ to $\cB'$,
then $F$ induces a morphism of fibre sequences
\[
  \xymatrix{
    {\Ne{\Qu{\cB}}_\bullet} \ar[r] \ar[d]_{\Ne{\Qu{F}}} &
    {\Ne{\Qu{\cA}}_\bullet} \ar[r] \ar[d]_{\Ne{\Qu{F}}} &
    {\Ne{\Qu{(\cA/\cB)}}_\bullet} \ar[d]^{\Ne{\Qu{F}}} \\
    {\Ne{\Qu{\cB'}}_\bullet} \ar[r] &
    {\Ne{\Qu{\cA'}}_\bullet} \ar[r] &
    {\Ne{\Qu{(\cA'/\cB')}}_\bullet.} \\
  }
\]

\begin{thm}\label{thmK}
  For an internal exact category $\cM$ and $i\in\Nn$,
  there are canonical group homomorphisms
  $\sa{}:\K{i}{\cM}\longrightarrow\sK{i}{\cM}$
  which are functorial in $\cM$ in the following sense:
  \begin{itemize}
    \item
      If $F:\cM\longrightarrow\cM'$ is an internal exact functor
      between internal exact categories,
      then the diagram
      \[
        \xymatrix@C=20mm{
          {\sK{n}{\cM}} \ar[r]^{\sK{n}{F}} & {\sK{n}{\cM'}} \\
          {\K{n}{\cM}} \ar[u]^{\sa{}} \ar[r]_{\K{n}{F}} & {\K{n}{\cM'}} \ar[u]_{\sa{}} \\
        }
      \]
      commutes for every $n\in\Nn$.
    \item
      If $\cB$ is a *Serre subcategory of an internal abelian category $\cA$,
      then the diagram
      \[
        \xymatrix@C=20mm{
          {\sK{n+1}{\cA/\cB}} \ar[r]^\delta & {\sK{n}{\cB}} \\
          {\K{n+1}{\cA/\cB}} \ar[u]^{\sa{}} \ar[r]_\delta & {\K{n}{\cB}} \ar[u]_{\sa{}} \\
        }
      \]
      commutes for all $i\in\Nn$.
  \end{itemize}
\end{thm}

\begin{proof}
  By \ref{lemmaQuillen} and \ref{lemmanerve},
  the simplicial set $\Ne{\Qu{\cM}}$
  is functorially isomorphic to $\res\ \sNe{\sQu{\cM}}$,
  so we can apply \ref{thmpi} and define $\sa{}$ by
  \[
    \sa{}:\
    \K{n}{\cM}=
    \pi_{n+1}(\Ne{\Qu{\cM}},0)=
    \pi_{n+1}(\res\ \sNe{\sQu{\cM}},0)\xrightarrow{\varphi_{n+1}}
    \str{\pi}_{n+1}(\sNe{\sQu{\cM}},0)=
    \sK{n}{\cM}
  \]
  for $n\in\Nn$.
  The functoriality of $\sa{}$ follows from the functoriality of $\varphi_n$
  proven in \ref{thmpi}.
\end{proof}

\begin{cor}\label{corSch}
  Let $k$ be an internal field.
  Then there is a canonical morphism of functors
  \[
    \s{}:\
    \K{n}{}\longrightarrow
    \sK{n}{}\circ\s{}:\
    \op{(\Schfp{k})}\longrightarrow\mathcal{Ab}
  \]
  for all $n\in\Nn$,
  i.e. for every $k$-scheme of finite type $X$ and every $n\in\Nn$,
  there is a canonical group homomorphism
  \[
    \s{}:\
    \K{n}{X}\longrightarrow
    \sK{n}{\s{X}}
  \]
  which is functorial in $X$.
\end{cor}

\begin{proof}
  For a $k$-scheme of finite type $X$,
  denote the category of vector bundles of finite rank on $X$ by $\VB{X}$.
  Then $\langle\VB{X},\cE\rangle$ is an exact category with
  \[
    \cE:=\left\{(i,j)\
    \left\vert\
    \mbox{$0\longrightarrow A'\xrightarrow{i}A\xrightarrow{j}B''\longrightarrow 0$
    exact in $\Coh{X}$}
    \right.\right\},
  \]
  and by definition $\K{n}{X}:=\K{n}{\VB{X}}$.
  Let $\sVB{X}$ be the internal exact category of *vector bundles of *finite rank on $\s{X}$.
  We know from \cite[5.5]{enlsch} that $\s{}:\Coh{X}\longrightarrow\sCoh{\s{X}}$ restricts to an exact functor
  $\s{}:\VB{X}\longrightarrow\sVB{\s{X}}$,
  and if $f:X'\longrightarrow X$ is a morphism of $k$-schemes of finite type,
  then the functors
  $\s{}\circ f^*$ and $[\s{f}]^*\circ\s{}$ from $\VB{X}$ to $\sVB{\s{X'}}$ are obviously canonically isomorphic.

  This implies that we get functorial group homomorphisms
  $\K{n}{X}\longrightarrow\K{n}{\sVB{\s{X}}}$ for all $n\in\Nn$,
  so we can define $\s{}$ as the composition
  \[
    \s{}:\
    \K{n}{X}\longrightarrow
    \K{n}{\sVB{\s{X}}}\xrightarrow{\sa{}}
    \sK{n}{\sVB{\s{X}}}=
    \sK{n}{\s{X}}.
  \]
  The functoriality of $\sa{}$ proven in \ref{thmK} shows that $\s{}$ thus defined is functorial.
\end{proof}


\bibliographystyle{alpha}
\bibliography{../Literatur}

\end{document}